\documentclass[12pt]{amsart}

\usepackage{amsfonts, amssymb, amsmath, amscd, amsthm, txfonts, graphicx, color, enumerate}

 \allowdisplaybreaks

\theoremstyle{plain}
\newtheorem{teo}{Theorem}[section]
\newtheorem{propo}[teo]{Proposition}
\newtheorem{lem}[teo]{Lemma}

\theoremstyle{definition}
\newtheorem{defin}[teo]{Definition}
\newtheorem{ej}[teo]{Example}

\numberwithin{equation}{section}

\theoremstyle{remark}
\newtheorem{obs}[teo]{Remark}

\usepackage{tikz}

\begin{document}

\title[Stable and Unstable Manifolds]{Local Stable and Unstable Manifolds for Anosov Families}

\author[J. Muentes]{Jeovanny de Jesus Muentes Acevedo}
\address{Instituto de Matem\'atica e Estat\'istica\\ Universidade de S\~ao Paulo\\
05508-090, Sao Paulo, Brazil}
\email{jeovanny@ime.usp.br}

\date{2017}

\begin{abstract}
Anosov families were introduced by A. Fisher and P. Arnoux motivated by generalizing the notion of   Anosov diffeomorphism defined on a compact Riemannian manifold.   In addition to presenting several properties and examples of   Anosov families, in this paper   we   build    local stable  and local manifolds for such families.  
\end{abstract}



\maketitle

\section{Introduction}

The Anosov families, which will be presented in Definition  \ref{anosovfamily}, were introduced by P. Arnoux and A. Fisher in  \cite{alb}, motivated by generalizing the notion of   Anosov diffeomorphisms. An Anosov family is a two-sided sequence   of diffeomorphisms    $f_{i}: M_{i}\rightarrow M_{i+1}$   defined on a sequence of  compact Riemannian manifolds $M_{i}$, for $i\in\mathbb{Z}$, having a similar behavior to an Anosov diffeomorphisms: the tangent bundle has a splitting    $TM_{i}=E^{s}\oplus E^{u}$, invariant by the derivative  $D ( f_{i+n}\circ\cdots\circ f_{i})$, and   
 there exist  constants $\lambda \in (0,1)$ and $c>0$  such that for   $n\geq 1$,  $p\in M_{i}$, 
 we have: \(\Vert D (f_{i+n-1}\circ...\circ f_{i})_{p} (v)\Vert \leq c\lambda^{n}\Vert v\Vert\) for     \(v\in E_{p}^{s}\) 
 and
 \(\Vert D (f_{i-n}^{-1}\circ...\circ f_{i-1}^{-1})_{p}(v)\Vert \leq c\lambda^{n}\Vert v\Vert\)  for    \(v\in E_{p}^{u}.\)     
 The  subspaces $E^{s}$ and $E^{u}$ are called     stable and unstable subspaces, respectively.   
 The main objective of this work is to obtain a   local stable and unstable manifold theorem for Anosov families. 

\medskip

 In the next section we  introduce the notion of Anosov family and, moreover, we present some examples of such families.       Readers may find, for example, in       \cite{alb}, \cite{Jeo2}, \cite{Jeo3} and \cite{Mikko},  several approaches and results in non-stationary   dynamic which have a  hyperbolic  behavior.    It is worth noting
that  it is not necessary that the $f_{i}$   be
  an Anosov diffeomorphism for  the family $(f_{i})_{i\in\mathbb{Z}}$ to be Anosov (see   \cite{alb}, Example 3). Other   interesting   examples can be obtained from    \textit{random hyperbolic dynamic systems}  (see  \cite{Gundlach}, \cite{Arnold}) or from \textit{hyperbolic linear cocycle linear} (see \cite{Jeo2}, \cite{luisb}).  
We will finish this section by presenting a notion of stable and unstable sets which works for families of diffeomorphisms (see Definition \ref{conjuntosestaviesfam}).   The stable (unstable) set at a point $p \in M_{i}$ consists of the  points $q\in M_{i}$  whose  (negative) positive orbit   approach exponentially to the  (negative) positive orbit   of $p$.
    
\medskip
 
In   Theorem \ref{primerpropovariedade} we will show    a  generalized version of Hadamard-Perron  Theorem to obtain   admissible manifolds    (see \cite{luisb}, \cite{Katok}). In our case,   stable and unstable subspaces of an Anosov family are not necessarily orthogonal. Additionally, the size of the submanifolds to be obtained  at a given point in the total space could decay along the orbit of such point.   These admissible manifolds  do not necessarily coincide with the stable or unstable subsets of a sequence of diffeomorphisms.

\medskip

 We will finish this work in the Section 4 
   with   Theorems \ref{variedadeinstave} and \ref{variedadeestavel}, the unstable and stable manifold Theorems for Anosov family. In these theorems    
     we give conditions with which the submanifolds obtained in   Section 3 coincide with the stable and unstable subsets  for an Anosov family, showing the uniqueness of the manifolds.       The results to be given here can be adapted to obtain stable and unstable manifolds for single hyperbolic maps,    non-uniform hyperbolic dynamical systems, random hyperbolic dynamical systems, among others systems (see \cite{luisb}, \cite{Gundlach},  \cite{Arnold}).  

\medskip

\section{Anosov Families}

 In this section, in addition to introduce the definition of Anosov family, we will   give 
some examples. Indeed, given a sequence of  Riemannian  manifolds  $M_{i}$, with fixed Riemannian metrics $\langle \cdot, \cdot\rangle_{i}$  for $i\in \mathbb{Z}$, consider the  \textit{disjoint union}  $$\textbf{M}=\coprod_{i\in \mathbb{Z}}{M_{i}}=\bigcup_{i\in \mathbb{Z}}{M_{i}\times{i}}.$$     
The set $\textbf{M}$ will be called   \textit{total space} and the $M_{i}$ will be called  \textit{components}.  We give the total space      the Riemannian metric   $\langle \cdot, \cdot\rangle$      defined as 
\( 
\langle \cdot, \cdot\rangle|_{M_{i}}=\langle \cdot, \cdot\rangle_{i}\),    for \(i\in \mathbb{Z}.\)   
We denote by $\Vert \cdot\Vert_{i}$ the  norm induced by  $\langle\cdot,\cdot\rangle_{i}$ on $TM_{i}$ and we will take   $\Vert \cdot \Vert$ defined on  $\textbf{M}$  as    $\Vert \cdot\Vert|_{M_{i}}=\Vert \cdot\Vert_{i} $ for $i\in \mathbb{Z}$.

\begin{defin}\label{leidecomposicao} A  \textit{non-stationary dynamical system} (or \textit{n.s.d.s.}) $(\textbf{M},\langle\cdot,\cdot\rangle, \textbf{\textit{f}})$  is an application $\textbf{\textit{f}}:\textbf{M}\rightarrow \textbf{M}$, such that, for each $i\in\mathbb{Z}$, $\textbf{\textit{f}}|_{M_{i}}=f_{i}:M_{i}\rightarrow M_{i+1}$ is a  $C^{1}$-diffeomorphism. Sometimes we use the notation   $\textbf{\textit{f}}=(f_{i})_{i\in\mathbb{Z}}$. A $n$-th composition is defined, for $i\in \mathbb{Z}$,   as   
\begin{equation*}
   \textbf{\textit{f}}_{ i} ^{n}:= \begin{cases}
  f_{i+n-1}\circ \cdots\circ f_{i}:M_{i}\rightarrow M_{i+n} & \mbox{if }n>0 \\
  f_{i-n}^{-1}\circ \cdots\circ f_{i-1}^{-1}:M_{i}\rightarrow M_{i-n} & \mbox{if }n<0 \\
	I_{i}:M_{i}\rightarrow M_{i} & \mbox{if }n=0,
        \end{cases}
\end{equation*}
\end{defin}

 This concept is known as \textit{sequences of mappings}, \textit{family of diffeomorphisms} or \textit{non-autono\-mous dynamical systems} (see   \cite{Jeo1}, \cite{Jeo2}, \cite{Jeo3}, \cite{alb}, \cite{Mikko}, and references there).   It can be   built a topological entropy for these systems. In \cite{Jeo1} we prove the continuity of this entropy.

\medskip

Since  $f_{i}$ is a   diffeomorphism,   the components $M_{i}$ are diffeomorphic  Riemannian manifolds. These  components could be, for instances, the same manifold with   Riemannian metrics $\langle\cdot,\cdot\rangle_{i}$ changing  with $i$ (see Figure \ref{Toros2}), or the  $M_{i}$'s could  be the same surface with    different \textit{fractal structures},  or with \textit{Thurston corrugations}, etc. (see   
 \cite{Borrelli2}).

  \begin{figure}[ht]
\begin{center}
\begin{tikzpicture}
\tikzstyle{point}=[circle,thick,draw=black,fill=black,inner sep=0pt,minimum width=1pt,minimum height=1pt]
\newcommand*{\xMin}{0}%
\newcommand*{\xMax}{6}%
\newcommand*{\yMin}{0}%
\newcommand*{\yMax}{6}%

\draw (-6.3,0.2) node[below] {\small \dots};
  
\draw (-6,0) .. controls (-5.9,0.8) and (-3.1,0.8) .. (-3,0);
\draw (-5.86,-0.01) .. controls (-5.85,0.6) and (-3.15,0.6) .. (-3.14,-0.01);
\draw (-6,0) .. controls (-5.9,-0.8) and (-3.1,-0.8) .. (-3,0);
\draw (-5.88,0.05) .. controls (-5.85,-0.6) and (-3.15,-0.6) .. (-3.12,0.05);
\draw (-4.5, -0.9) node[below] {\small $M_{i-1}$};

\draw (-2.5,0.45) node[below] {\small $\xrightarrow{f_{i-1}}$};

\draw (-2,0) .. controls (-2.05,0.8) and (0.1,0.8) .. (0,0);
\draw (-2,0) .. controls (-2.05,-0.8) and (0.1,-0.8) .. (0,0);
\draw (-1.5,-0.029) .. controls (-1.31,0.2) and (-0.7,0.2) .. (-0.5,-0.03);
\draw (-1.55,0.05) .. controls (-1.45,-0.2) and (-0.55,-0.2) .. (-0.45,0.05);

\draw (-1, -0.9) node[below] {\small $M_{i}$};

\draw (0.5,0.45) node[below] {\small  $\xrightarrow{\,\, \, f_{i}\, \, \,}$};

 \draw (1,0) .. controls (1,-0.7) and (3.5,-1.5) .. (3.5,0);
 \draw (1.2,0.1) .. controls (1,-0.3) and (3,-0.1) .. (3,0.4);

 \draw (1,0) .. controls (1,0.7) and (3.5,1.5) .. (3.5,0);
 \draw (1.2,0) .. controls (1.1,0.3) and (3,0.5) .. (2.8,0.15);
\draw (2.5, -0.9) node[below] {\small $M_{i+1}$};

 \draw (3.9,0.2) node[below] {\small \dots};
\end{tikzpicture}
\end{center}
\caption{A n.s.d.s. on a sequence of 2-torus endowed with different Riemannian metrics.}
\label{Toros2}
\end{figure}
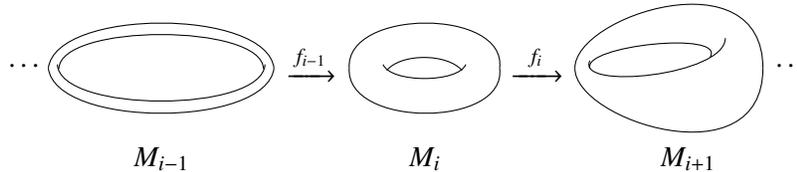

\begin{defin}\label{anosovfamily}  A n.s.d.s.     $(\textbf{M},\langle\cdot,\cdot\rangle, \textbf{\textit{f}})$  is called an \textit{Anosov family} if:
\begin{enumerate}[\upshape (i)]
\item the tangent bundle $T\textbf{M}$ has a continuous splitting   $E^{s}\oplus E^{u}$ which is  $D\textbf{\textit{f}}$-\textit{invariant}, i. e., for each $p\in \textbf{M}$, 
 $T_{p}\textbf{M}=E^{s}_{p}\oplus E^{u}_{p}$ with $D \textbf{\textit{f}}_{p}(E^{s}_{p})= E^{s}_{\textbf{\textit{f}}(p)}$ and $D \textbf{\textit{f}}_{p}(E^{u}_{p})= E^{u}_{\textbf{\textit{f}}(p)}$, where $T_{p}\textbf{M} $ is the    tangent space at $p;$
\item there exist constants $\lambda \in (0,1)$ and $c>0$ such that for each  $i\in \mathbb{Z}$, $n\geq 1$,    and $p\in M_{i}$, 
we have: \[\Vert D (\textbf{\textit{f}}_{i}^{n})_{p}(v)\Vert \leq c\lambda^{n}\Vert v\Vert \text{ if   }v\in E_{p}^{s}\quad\text{and}\quad \Vert D (\textbf{\textit{f}}_{i}^{-n})_{p} (v)\Vert \leq c\lambda^{n}\Vert v\Vert  \text{ if }v\in E_{p}^{u}.\]
\end{enumerate}
The  subspaces $E^{s}_{p}$ and $E^{u}_{p}$ are called     stable and unstable subspaces, respectively. If we can take
$c=1$ we say the family is \textit{strictly Anosov}. See Figure \ref{Anosovdifeo}.
\end{defin}
   
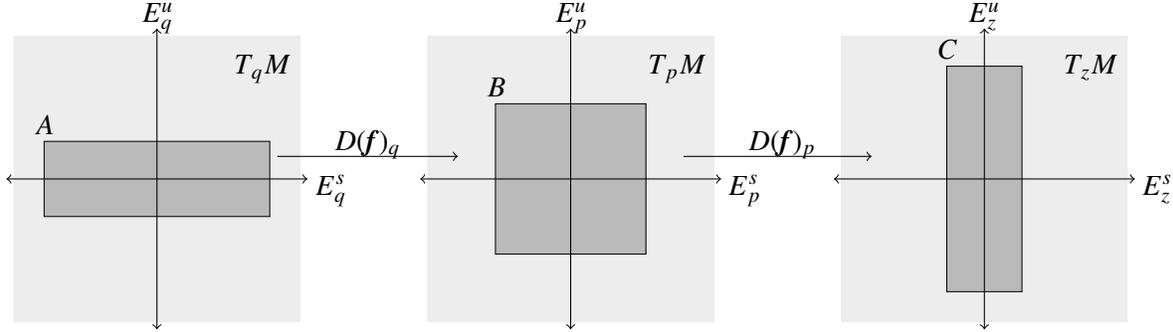
\begin{figure}[ht] 
\begin{center}

\begin{tikzpicture}
\draw[black!7,fill=black!7, ultra thin] (-7.9,-0.4)rectangle (-4.1,3.4);
\draw[black!7,fill=black!7, ultra thin] (-2.4,-0.4)rectangle (1.4,3.4);
\draw[black!7,fill=black!7, ultra thin] (3.1,-0.4)rectangle (6.9,3.4);

\draw[black,fill=black!27, very thin] (-7.5,1)rectangle (-4.5,2);
 \draw[<->] (-6,-0.5) -- (-6,3.5);
 \draw[<->] (-8,1.5) -- (-4,1.5); 
\draw (-6.2,4) node[below] {\quad{\small $E_{q}^{u}$}}; 
\draw (-3.9,1.7) node[below] {\quad{\small $E_{q}^{s}$}};
\draw (-4.6,3.3) node[below] {\small $T_{q}M$};
\draw (-3.2,2.3) node[below] {\small $D(\textbf{\textit{f}})_{q}$};
 \draw[->] (-4.4,1.8) -- (-2,1.8);
 \draw (-7.7,2.5) node[below] {\quad{\small $A$}};
 
\draw[black,fill=black!27, very thin] (-1.5,0.5)rectangle (0.5,2.5);
 \draw[<->] (-0.5,-0.5) -- (-0.5,3.5);
 \draw[<->] (-2.5,1.5) -- (1.5,1.5);
\draw (-0.7,4) node[below] {\quad{\small $E_{p}^{u}$}};
\draw (1.6,1.7) node[below] {\quad{\small $E_{p}^{s}$}};
\draw (0.7,3.3) node[below] {\quad{\small $T_{p}M$}};
\draw (2.1,2.3) node[below] {\quad{\small $D(\textbf{\textit{f}})_{p}$}};
\draw[->] (1.0,1.8) -- (3.5,1.8);
\draw (-1.7,3) node[below] {\quad{\small $B$}};

\draw[black,fill=black!27, very thin] (4.5,0.0)rectangle (5.5,3);
 \draw[<->] (5,-0.5) -- (5,3.5);
 \draw[<->] (3,1.5) -- (7,1.5);
 \draw (5,4) node[below] {\small $E_{z}^{u}$};
 \draw (7.3,1.7) node[below] {\small $E_{z}^{s}$};
 \draw (6.2,3.3) node[below] {\quad{\small $T_{z}M$}};
 \draw (4.3,3.5) node[below] {\quad{\small $C$}};
\end{tikzpicture} 

\end{center}
\caption{$q=\textbf{\textit{f}}^{-1}(p)$ and $r=\textbf{\textit{f}}(p)$. $D(\textbf{\textit{f}})_{q}(A)=B$ and $D(\textbf{\textit{f}})_{p}(B)=C$}
\label{Anosovdifeo}
\end{figure}

In \cite{Jeo2} we proved the set consisting of Anosov families is open in the set consisting of n.s.d.s. on \textbf{M}, endowed with the Whitney topology (or strong topology).  The structural stability of certain families is studied in \cite{Jeo3}. 
 
\medskip 

  The splitting $T\textbf{M}=E^{s}\oplus E^{u}$ induced by an Anosov family is unique  (see \cite{alb}, Proposition 2.12). Actually, in \cite{Jeo2},  we prove    for each $p\in M_{i}$,     \begin{align*}&E_{p} ^{s}=\{v\in T_{p}M_{i}: ( \Vert D (
\textbf{\textit{f}}^{n}_{i})_{p}  (v)\Vert)_{n\geq1}\text{ is bounded }\} \\
 \text{and }\quad\quad&E_{p} ^{u}=\{v\in T_{p}M_{i}: (\Vert D (\textbf{\textit{f}}^{-n}_{i}) _{p} (v)\Vert)_{n\geq1}\text{ is bounded }\}.
 \end{align*}  

It is clear that, if $M$ is a compact Riemannian manifold with Riemannian metric $\langle\cdot,\cdot\rangle$, $M_{i}=M\times\{i\}$ endowed with the metric $\langle\cdot,\cdot\rangle_{i}= \langle\cdot,\cdot\rangle $,   and $f_{i}$ is an Anosov diffeomorphism on $M$, for $i\in \mathbb{Z}$,   then $(f_{i})_{i\in\mathbb{Z}}$  is an Anosov family.  
 
 \medskip 

The notion of Anosov diffeomorphism on a compact Riemannian manifold does not depend on the
Riemannian metric (see \cite{Katok}). In the case of n.s.d.s., by suitably changing the metric
 $\langle\cdot,\cdot\rangle_{i}$ on each $M_{i}$, the \textit{constant family   associated} to the identity   could become  
an Anosov family (see \cite{alb}, Example 4). 
Hence, it is important to keep fixed the metrics on each  $M_{i}$.

  \medskip

A homeomorphism $\psi:X\rightarrow X$  on the metric space $(X,d)$ is \textit{expansive} in
a  subset $Y$ of $X$  if there is $\varepsilon>0$  such that for each $y\in Y$, $x\in X$, with $x\neq y$, there exists $n\in\mathbb{Z}$ such that  $d(\psi^{n}(x),\psi^{n}(y)) >\varepsilon.$  
It is well known that if  $\Lambda\subseteq M$ is a compact hyperbolic subset for a $C^{1}$-diffeomorphism
 $\phi:M\rightarrow M$, then $\phi$ is expansive on $\Lambda.$ In the following example we will see that there are
Anosov families that are not expansive.

\begin{ej}\label{contraexemplo} Let  $M$ be a Riemannian manifold with Riemannian norm $\Vert\cdot\Vert$ and $\phi:M\rightarrow M$   an
Anosov diffeomorphism with constants $c\geq1$ and $\lambda\in(0,1).$ Take $M_{i}=M$ for all $i$ with
  Riemannian norm defined as    
\begin{equation}\label{metricatotal234}
\Vert (v_{s},v_{u})\Vert_{i}=    \begin{cases}       
        \sqrt{a^{2i}\Vert v_{s}\Vert^{2}+b^{2i}\Vert v_{u}\Vert^{2}}& \mbox{if }i\geq0 \\
        	\Vert (v_{s},v_{u})\Vert& \mbox{if }i<0, \\
        \end{cases}
\end{equation}
where  $a,b\in(\lambda,1/\lambda) $. Consider  $\textbf{M}$ as the disjoint union of the $M_{i}$ with the norm \eqref{metricatotal234}. Let $f_{i}:M_{i}\rightarrow M_{i+1}$ defined as  $f_{i}(x,i)=(\phi(x),i+1)$ for $x\in M$, $i\in\mathbb{Z}$.  It is not difficult to prove  that $\textbf{\textit{f}}=(f_{i})_{i\in\mathbb{Z}}$ is an   Anosov family  on $\textbf{M}$ with  constants $\tilde{c}=c$ and $\tilde{\lambda}=\max\{\lambda, a\lambda,\lambda/b\}<1$, where  the  splitting of    $T\textbf{M}$ is the same induced  by   $\phi$. Notice that if  $a,b\in (\lambda, 1)$, then, for $x,y\in M_0$ we obtain $d(\textbf{\textit{f}}^{n}(x),\textbf{\textit{f}}^{n}(y))\rightarrow 0$ as   $n\rightarrow +\infty$\footnote{Notice that the volume of each $M_{i}$ with the Riemannian metric $\Vert \cdot\Vert_{i}$ defined in \eqref{metricatotal234} is  decreasing, for $i\geq1$, if $a,b\in (\lambda, 1)$ (see Figure \ref{torodecre}).}.   \begin{figure}[ht]
\begin{center}
\begin{tikzpicture}
\tikzstyle{point}=[circle,thick,draw=black,fill=black,inner sep=0pt,minimum width=1pt,minimum height=1pt]
\newcommand*{\xMin}{0}%
\newcommand*{\xMax}{6}%
\newcommand*{\yMin}{0}%
\newcommand*{\yMax}{6}%

\draw (-6.5,0.2) node[below] {\small \dots};
  \draw (-6,0) .. controls (-6.05,1.1) and (-2.95,1.1) .. (-3,0);
\draw (-6,0) .. controls (-6.05,-1.1) and (-2.95,-1.1) .. (-3,0);
\draw (-5.3,-0.01) .. controls (-5.1,0.33) and (-3.9,0.33) .. (-3.7,-0.01);
 \draw (-5.34,0.05) .. controls (-5,-0.33) and (-3.92,-0.33) .. (-3.66,0.05);
\draw (-4.5, -0.8) node[below] {\small $M_{1}$};
\draw (-2,0) .. controls (-2.05,0.8) and (0.05,0.8) .. (0,0);
\draw (-2,0) .. controls (-2.05,-0.8) and (0.05,-0.8) .. (0,0);
\draw (-1.5,-0.029) .. controls (-1.31,0.2) and (-0.7,0.2) .. (-0.5,-0.03);
\draw (-1.55,0.05) .. controls (-1.45,-0.19) and (-0.55,-0.19) .. (-0.45,0.05);
\draw (-1, -0.8) node[below] {\small $M_{2}$};
 \draw (1,0) .. controls (1,-0.6) and (2.5,-0.6) .. (2.5,0);
 \draw (1,0) .. controls (1,0.6) and (2.5,0.6) .. (2.5,0);
 \draw (1.35,0) .. controls (1.5,0.18) and (2,0.18) .. (2.15,0);
 \draw (1.3,0.03) .. controls (1.5,-0.18) and (2,-0.18) .. (2.2,0.03);
\draw (1.8, -0.8) node[below] {\small $M_{3}$};
\draw (3,0.2) node[below] {\small \dots};
\end{tikzpicture}
\end{center}
\caption{$M_{1}, M_{2}, M_{3}$,\dots,   endowed with the metric given in \eqref{metricatotal234},   for $a,b\in (\lambda,1)$.} \label{torodecre}
\end{figure}
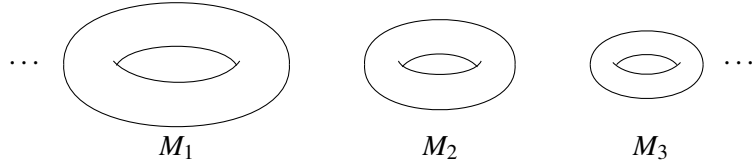

  On the other hand, if  $y$ belongs to the unstable submanifold of $\phi$ at $x$,   we obtain   $d(\textbf{\textit{f}} ^{-n}(x),\textbf{\textit{f}} ^{-n}(y))=d(\phi^{-n}(x),\phi^{-n}(y))\rightarrow 0$ as $n\rightarrow +\infty$. 
Consequently, \textbf{\textit{f}}  is not  expansive.
 \end{ej}

For each $i\in \mathbb{Z}$, let \begin{equation}\label{vainadelosangulos}\theta_{i}=\min_{p\in M_{i}} \{\theta_{p}:\theta_{p}\text{ is the angle between }E_{p}^{s}\text{ and }E_{p}^{u}\}
.\end{equation} 
The $\theta_{i}$'s  are uniformly
bounded away  from zero on each component (see \cite{Jeo2}).  
  We say that $(\textbf{M},\langle\cdot,\cdot\rangle,\textbf{\textit{f}})$ satisfies the  \textit{property of the angles} if there exists $ \mu\in(0,1)$  such that, for any $i\in\mathbb{Z}$,  we have \(\text{cos}(\theta_{i})\in [0,\mu].
\)
The following example  shows that there exist Anosov families that do not satisfy the property of the angles.

\begin{ej}\label{parelelizavelangulos} Let $M=\mathbb{T}^{2}$ and $\phi:M\rightarrow M$ be the   Anosov diffeomorphism induced by the  matrix \begin{equation*} A =
\left(
\begin{array}{ccc}
2 & 1 \\
1 & 1
 \end{array}
\right) .
\end{equation*}  
  The eigenvalues of $A$ are $\lambda=(3+\sqrt{5} )/2>1$ and $1/\lambda$.  
Consider  the  eigenvectors    $v_{s}=((1+\sqrt{5})/2,1)$ and $v_{u}=((1-\sqrt{5})/2,1)$    of $A$ associated to  $\lambda$ and $1/\lambda$, respectively.  
Let $(\zeta_{i})_{i\in\mathbb{Z}}$ be a sequence in  $[0,1).$ In the basis $\{v_{s},v_{u}\}$ of $\mathbb{R}^{2}$, set \begin{equation*} B_{i} =
\left(
\begin{array}{ccc}
1 & \zeta_{i} \\
\zeta_{i} & 1
 \end{array}
\right) \quad i\in\mathbb{Z}.
\end{equation*} 
The eigenvalues  of $B_{i}$ are $\alpha_{i}=1+\zeta_{i}$ and $\beta_{i}=1-\zeta_{i}.$ Since $\zeta_{i}\in [0,1)$, the matrix $B_{i}$ is positive definite. Thus, it induces an inner product    $\langle\cdot,\cdot\rangle_{i}$ on $\mathbb{R}^{2}$: if $v_{1}=av_{s}+bv_{u}$, $v_{1}=cv_{s}+dv_{u}\in  \mathbb{R}^{2}$, 
\begin{equation*} \langle v_{1},v_{2}\rangle_{i} =
 \left(
\begin{array}{ccc}
a & b
 \end{array}
\right) \left(
\begin{array}{ccc}
1 & \zeta_{i} \\
\zeta_{i} & 1
 \end{array}
\right) \left(
\begin{array}{ccc}
c  \\
d 
 \end{array}
\right) \quad i\in\mathbb{Z}.
\end{equation*} 
Notice that the angle between   $v_{s}$ and $v_{u}$ with the inner product $ \langle \cdot,\cdot\rangle_{i}$ is: 
\[\theta_{i}=\arccos\left(\frac{\langle v_{1},v_{2}\rangle_{i}}{\sqrt{\langle v_{1},v_{1}\rangle_{i}\cdot \langle v_{2},v_{2}\rangle_{i}}}\right)=\zeta_{i}.\] 
Furthermore, if $\Vert\cdot\Vert_{i}$  is the norm    induced by $ \langle \cdot,\cdot\rangle_{i}$  and $\Vert\cdot\Vert$ is the canonical  norm  of $\mathbb{R}^{2}$, we have $\Vert v_{s}\Vert_{i}=\Vert v_{s}\Vert$ and $\Vert v_{u}\Vert_{i}=\Vert v_{u}\Vert$ for all $i\in\mathbb{Z}$ (the inner product $ \langle \cdot,\cdot\rangle_{i}$ only change   the angles between $v_{s}$ and $v_{u}$). 
Consequently, $(\textbf{M},\langle\cdot,\cdot\rangle  ,\textbf{\textit{f}}) $ is an Anosov family, where  \textbf{M} is the disjoint union of the $M_{i}$, $\langle\cdot,\cdot\rangle $ is obtained by $\langle\cdot,\cdot\rangle_{i}  $ and $f_{i}(x,i)=(\phi(x),i+1)$ for $x\in M$, $i\in\mathbb{Z}$. If $\zeta_{i}\rightarrow 0$ as $i\rightarrow \infty$, then $(\textbf{M},\langle\cdot,\cdot\rangle  ,\textbf{\textit{f}}) $ is an Anosov family that does not satisfy the property of the angles.
\end{ej}

In \cite{Jeo2},   we show that there exists a Riemannian metric $\langle\cdot,\cdot\rangle_{\ast}$ on
 \textbf{M}, equivalent to $\langle\cdot,\cdot\rangle$ on each $M_{i}$, such that $(\textbf{M},\langle\cdot,\cdot\rangle_{\ast}, \textbf{\textit{f}})$ is a   strictly
Anosov family and satisfies the property of the angles. In the case of an Anosov  diffeomorphism
on  a compact Riemannian manifold the previous fact is known as \textit{Lemma of
Mather} and, by compactness, the metric $\langle\cdot,\cdot\rangle^{\ast}$ is uniformly equivalent to   $\langle\cdot,\cdot\rangle$.
In the case of families, the metric $\langle\cdot,\cdot\rangle_{\ast}$  is    uniformly equivalent to
$\langle\cdot,\cdot\rangle$ on \textbf{M} if and only if   $(\textbf{M},\langle\cdot,\cdot\rangle,\textbf{\textit{f}})$  satisfies the property of the angles (see \eqref{cosadeequiv}).

\medskip

 The stable  and unstable sets  for n.s.d.s. to be considered here  consist of the  points   whose   orbits   approach exponentially to the   orbit    of a given point.      Let $ d_ {i} (\cdot, \cdot) $  be the   Riemannian metric induced by $ \langle \cdot, \cdot \rangle_{i} $ on $M_{i}$.  To simplify notation,  we will use     $d (\cdot,\cdot)$ to denote that metric.  
   
   \begin{defin}\label{defindecai} Given two points $p,q\in \textbf{M}$, set
\begin{equation*} 
\Theta_{p,q}= \underset{n\rightarrow \infty}\limsup  \frac{1}{n}\log d(\textbf{\textit{f}}_{i}^{n}(q),\textbf{\textit{f}}_{i}^{n}(p))\quad \text{and}\quad \Omega_{p,q} =\underset{n\rightarrow \infty}\limsup\frac{1}{n}\log d(\textbf{\textit{f}}_{i}^{-n}(q),\textbf{\textit{f}}_{i}^{-n}(p)).
\end{equation*}
   \end{defin}
     
\begin{defin}\label{conjuntosestaviesfam} Let   $\varepsilon=(\varepsilon_{i})_{i\in\mathbb{Z}}$ be a sequence of positive numbers.   Fix  $p\in M_{i}$.   Let  $B(p,\delta )$ be the ball with center $p$ and radius $\delta>0$. Set 
\begin{enumerate}[\upshape (i)]
\item  \(\mathcal{N}^{s}(p,\varepsilon )=\{q\in M_{i}:
 \textbf{\textit{f}}_{i}^{n}(q)\in B(\textbf{\textit{f}}_{i}^{n}(p),\varepsilon_{i+n})\text{ for }n\geq0\text{ and }\Theta_{p,q} <0\}\):=  \textit{the local stable set at} $p$; 
\item  \(\mathcal{N}^{u}(p,\varepsilon)=\{q\in M_{i}:  
\textbf{\textit{f}}_{i}^{-n}(q)\in B(\textbf{\textit{f}}_{i}^{-n}(p),\varepsilon_{i-n})\text{ for  }n\geq1 \text{ and  
} \Omega_{p,q}<0\}\):=   \textit{the local unstable set at} $p$.
\end{enumerate}
\end{defin}

 In the Section 4, we will give conditions with which the local stable and unstable sets for Anosov families are submanifolds differentiable tangent  to stable  and unstable subspaces (see Theorems \ref{variedadeinstave} and \ref{variedadeestavel}). 

\medskip

 The existence of Anosov diffeomorphisms $\phi:M\rightarrow M$ imposes strong restrictions on the manifold $M$.   All known examples of Anosov diffeomorphisms are defined on \textit{infranilmanifolds} (see      
 \cite{luisb},   \cite{Katok}). If $M$ is a parallelizable Riemannian
  manifold, suitably changing the metrics on each component $M_{i}=M\times \{i\}$  
we can obtain an Anosov family on  $\textbf{M}$, taking  $f_{i}$ as
the identity $I_{i}:M_{i}\rightarrow M_{i+1}$ (see \cite{alb}, \cite{Jeo2}). An  Anosov family    does not   necessarily consist of Anosov diffeomorphisms.    A natural question that arises from the above is:  Let $M$ be a  parallelizable Riemannian manifold, with  Riemannian metric $\langle\cdot,\cdot\rangle.$  Take  $M_{i}=M\times \{i\}$ with   Riemannian metric   $\langle\cdot,\cdot\rangle_{i}=\langle\cdot,\cdot\rangle$ for all $i\in\mathbb{Z},$ and let  \textbf{M} be the  disjoint union of the $M_{i}$'s. Is there any  Anosov family  on \textbf{M}?
 Since the constant family associated to an  Anosov difeomorphism is an Anosov  family,    each manifold   admitting an Anosov diffeomorphism admits an Anosov family. 

\medskip

It is well-known that there are not Anosov diffeomorphisms on the circle $\mathbb{S}^{1}$.  Next we prove that the circle   does not admit  Anosov families. 

\begin{propo} Set  $M_{i}= \mathbb{S}^{1}\times \{i\}$ with   Riemannian  metric inherited from $\mathbb{R}^{2}$  and \textbf{M}  disjoint  union  of the  $M_{i}.$
Thus, there is not  any Anosov family on \textbf{M}.\end{propo}
\begin{proof} Suppose that $(f_{i})_{i\in\mathbb{Z}}$ is an Anosov family on \textbf{M}. Fix $p\in M_{0}$. Since the circle is one-dimensional, then, either $\Vert  D(\textbf{\textit{f}}_{0} ^{n})_{p}(v)\Vert\leq c\lambda^{n} \Vert v\Vert$ for all $n\geq 1$, $v\in T_{p}\mathbb{S}^{1}$ or $\Vert  D(\textbf{\textit{f}}_{0} ^{-n})_{p}(v)\Vert\leq c\lambda^{n}\Vert v\Vert$ for all $n\geq 1$, $v\in T_{p}\mathbb{S}^{1}$. Without loss of generality we can assume that $\Vert  D(\textbf{\textit{f}}_{0} ^{n})_{p}(v)\Vert\leq c\lambda^{n}\Vert v\Vert$ for all $n\geq 1$, $v\in T_{p}\mathbb{S}^{1}$. Let $n\in\mathbb{N}$ be   such that    $c\lambda^{n}<1/2.$ Take $\phi:\mathbb{S}^{1}\rightarrow \mathbb{S}^{1}$ as  $\phi=f_{n-1}\circ\cdots\circ f_{0}$.  If  $p\in \mathbb{S}^{1}$,  then $\Vert  D \phi_{p}(v)\Vert\leq(1/2)\Vert v\Vert$ for all   $v\in T_{p}\mathbb{S}^{1}$. Since  $\phi $ is a homeomorphism, it is impossible. 
\end{proof}

From the previous proposition we get that if $M$ is   $\mathbb{S}^{1}$ then the answer to the above question  is 
``no''. This fact leaves another question:
  Let \textbf{M} be the disjoint union of $M_{i}=M\times \{i\}$. Does \textbf{M} admits an  
Anosov family if and only if $M$ admits an Anosov diffeomorphism?

\section{Hadamard-Perron Theorem for Anosov Families}
If $\phi:M\rightarrow M$ is a diffeomorphism  on a  Riemannian manifold   $M$, and $\Lambda\subseteq M$ is a compact hyperbolic  set of $\phi$, then  there exists $\varepsilon >0$ such that, for all $x\in \Lambda,$ the \textit{local stable set} at $x$, denoted by $W_{\varepsilon} ^{s}(x)$, 
and the \textit{local unstable set} at $x$,  denoted by $ W^{u}_{\varepsilon}(x)$, 
are differentiable submanifolds of $M$, 
tangent to the stable and unstable subspaces     at $ x$, respectively (see  \cite{Morris}).  In that case,     $\phi$ is a contraction on $W_{\varepsilon} ^{s} (x)$       (that is, there exists   $\nu \in (0,1)$  such that $d(\phi(z),\phi(y))\leq \nu d(z,y)$  for all $z,y \in W_{\varepsilon} ^{s} (x)$) and  $\phi^{-1}$ is a contraction on $W_{\varepsilon} ^{u} (x)$. Furthermore, $\phi(W_{\varepsilon} ^{s} (x))\subseteq W_{\varepsilon} ^{s} (\phi(x))$ and $\phi^{-1}(W_{\varepsilon} ^{u} (\phi(x)))\subseteq W_{\varepsilon} ^{u} (x)$ for each $x\in M$.  The facts above   are not always   valid for Anosov families, neither considering stable (unstable) sets for homeomorphisms (see Example \ref{contraexemplo}) nor considering stable (unstable) sets for n.s.d.s., because  it is not always possible to find a sequence of positive numbers $ \delta_{i} $ such that, for all $i$,    $f_{i}$  and its derivative $ Df_{i}  $, restricted to   balls of radius  $ \delta_{i} $, have the same qualitative behavior (see   \eqref{deltamcrecimiento}).  In this section we will give conditions to obtain invariant manifolds at each point of the total space, whose expansion or contraction by each $ f_i $  can be controlled (see Theorems \ref{primerpropovariedade}  and \ref{primerpropovariedade2}). This result is a generalized  version of Hadamard-Perron  Theorem (as well known as \textit{Pesin theory}) to build local stable and unstable manifold   for Anosov families (see \cite{luisb}, \cite{Katok}). In our case,   stable and unstable subspaces are not necessarily orthogonal  and the size of the manifolds to be obtained here could decrease along the orbits (see \eqref{deltamcrecimiento}). 

\medskip

 We will fix an Anosov family $(\textbf{M},\langle \cdot,\cdot\rangle,\textbf{\textit{f}})$    with constant   $\lambda\in(0,1)$ and $c\geq1$. 
   
 \begin{obs}If $c>1$,   we will consider   a gathering   of   \textbf{\textit{f}} instead of \textit{\textbf{f}}: we say that    $\tilde{\textbf{\textit{f}}}$  is a \textit{gathering  of   \textbf{\textit{f}}   with  lenght} $n\in\mathbb{N}$ if     $\tilde{\textbf{\textit{f}}}_{i}=f_{n(i+1)-1}\circ \cdots \circ f_{ni+1}\circ f_{ni}=\textbf{\textit{f}}_{ni}^{\, n}$ for each $i\in\mathbb{Z}$:  
\begin{equation*}\begin{CD}
\cdots M_{n(i-1)}@>{\tilde{f}_{i-1}=f_{ni-1}\circ \cdots  \circ f_{n(i-1)}}>> M_{ni}@>{\tilde{f}_{i}=f_{n(i+1)-1}\circ \cdots  \circ f_{ni}}>>M_{n(i+1)}   \cdots
\end{CD} 
\end{equation*}
Let $n $ be the minimum positive integer such that $c\lambda^{n}\leq\lambda$. Hence  the gathering  $\tilde{\textbf{\textit{f}}}$ with lenght $n$ is a strictly Anosov family with constant $\lambda$. Thus, considering a gathering of \textbf{\textit{f}} if necessary,   we can assume that the family is strictly Anosov.  
\end{obs}

    Let us fix $p\in \textbf{M}$.  
Without loss of generality, we can assume  $p\in M_{0}$ (if $p\notin M_{0}$, $q=\textbf{\textit{f}}^{n}(p)\in M_{0}$ for some  $n\in \mathbb{Z}$, then consider    $q$ instead of $p$). To simplify the notation, given $\varepsilon>0$, let $B_{n}(\varepsilon)\subseteq T_{\textbf{\textit{f}}_{0} ^{n}(p)} \textbf{M}$  be the ball with radius $\varepsilon $ and center $0$;  
  $B^{s}_{n}(\varepsilon)\subseteq E_{f^{n}(p)}^{s}  $   the ball with radius $\varepsilon $ and center $ 0$;  $B^{u}_{n}(\varepsilon) \subseteq E_{f^{n}(p)}^{u} $   the ball with radius $\varepsilon $ and center $ 0$.  
  
\medskip

 The following subspaces will be very useful    to prove  Proposition \ref{funciongrafico}.  Let    $\alpha \in(0,1)$ and  $(\gamma_{n})_{n\in\mathbb{Z}}$ be a sequence of positive  numbers. Set:
\begin{enumerate}[\upshape (i)] 
\item    $\Gamma^{u}_{n}(\alpha,\gamma_{n}) =\{\phi :B_{n}^{u}(\gamma_{n}) \rightarrow B_{n}^{s}(\gamma_{n}):  \phi \text{ is }\alpha\text{-Lipschitz and }     \phi(0)=0\}$.
\item    $\Gamma^{u} (\alpha,(\gamma_{n})_{n}) =\{\phi=(\phi_{n})_{n\in \mathbb{Z}} :  \phi_{n}\in \Gamma^{u}_{n}(\alpha,\gamma_{n})\}$.
\end{enumerate}
If $\phi=(\phi_{n})_{n\in \mathbb{Z}}, \psi=(\psi_{n})_{n\in \mathbb{Z}}\in \Gamma^{u} (\alpha,(\gamma_{n})_{n})$, define the metric \[d_{\Gamma^{u}}(\phi ,\psi)=\sup_{n\in \mathbb{Z}}\left\{\sup_{x\in B_{n}^{u}(\gamma_{n})\setminus \{0\}}\frac{\Vert \phi_{n}(x)-\psi_{n}(x)\Vert}{\Vert x\Vert}\right\}.\]
Hence, we have:
\begin{propo}  $(\Gamma^{u} (\alpha,(\gamma_{n})_{n}), d_{\Gamma^{u}})$  is a complete metric space.
\end{propo}
For
a map   $F:X\rightarrow Y$, we will denote by $\mathcal{G}(F)$  the set $\{(F(x),x):x\in X\}$. Notice that, if $\phi\in \Gamma^{u}_{n}(\alpha,\gamma_{n}),$ then \[ \mathcal{G}(\phi)\subseteq  \{(v,w)\in B_{n}^{s}(\gamma_{n})\times B_{n}^{u}(\gamma_{n}):\Vert v\Vert \leq \alpha\Vert w\Vert \}.\] 
  
  For each $n\in\mathbb{Z}$, let $\varepsilon_{n}>0$ be such that the exponential application   \[\text{exp}_{\textbf{\textit{f}}_{0} ^{n}(p)}:B_{n}(\varepsilon_{n}) \rightarrow B(\textbf{\textit{f}}_{0} ^{n}(p),\varepsilon_{n}) \] is a diffeomorphism and \( \Vert v\Vert  =d(\text{exp}_{\textbf{\textit{f}}_{0} ^{n}(p)}(v), \textbf{\textit{f}}_{0} ^{n}(p)),\)  for all  \(v\in B_{n}(\varepsilon_{n}),
\)
 that is, $\varepsilon_{n}$ is the  \textit{injectivity radius} of $\text{exp}_{\textbf{\textit{f}}_{0} ^{n}(p)}$ at $\textbf{\textit{f}}_{0} ^{n}(p)$, which we denote by $r(\textbf{\textit{f}}_{0} ^{n}(p))$. 
 Now, take  $ \alpha = (\lambda ^{-1}-1)/2 $ and let $\delta_{n}>0$   be small  enough such that 
   \begin{equation*}\tilde{f}_{n} =\text{exp}^{-1} _{f^{n+1}(p)}\circ f_{n}\circ \text{exp}_{f^{n}(p)}:B_{n}(\delta_{n})\rightarrow B_{n+1}(\varepsilon_{n+1})
\end{equation*}  
is well defined, for each $n$. It is clear that   $\delta_{n}$ depends on both $\varepsilon_{n}$ and  $f_{n}$.  

\begin{obs} For each $n\in \mathbb{Z}$, consider  $L_{n}=\max_{p\in M_{n}}\Vert D (f_{n})_{p}\Vert$. Notice that  if, for each $n$,   $\delta_{n}\leq \min \{\epsilon_{n},\epsilon_{n+1}/\max\{L_{n},1\}\},$  then, for all $(v,w)\in B^{s}_{n}(\delta_{n}) \times B^{u}_{n}(\delta_{n})$, we have $\tilde{f}_{n}(v,w)\in B_{n+1}(\varepsilon_{n+1})$.  Consequently, if $M_{n}=M\times \{n\}$, where $M$ is a compact Riemannian manifold, $\langle\cdot,\cdot\rangle_{n}=\langle\cdot,\cdot\rangle,$ where $\langle\cdot,\cdot\rangle$ is the Riemannian metric on $M$, and  $(L_{n})_{n\in\mathbb{Z}}$ is bounded,  then we can find a  uniform  $\delta $ with which $ \tilde{f}_{n}$ is well-defined for each $n\geq0,$ that is, there exists $\delta>0$ such that, considering $\delta_{n}=\delta$ for each $n\geq0$,  $\tilde{f}_{n}$ is well-defined.
\end{obs}
 
If   $(v,w)\in B_{n}(\delta_{n})$, with $v\in E^{s}$ and $w\in E^{u}$, then 
\begin{align*}\tilde{f}_{n}(v,w)=(a_{n}(v,w)+F_{n}(v) ,b_{n}(v,w) +F_{n}(w)),
\end{align*}
where   \( 
F_{n}  =D (f_{n})_{\textbf{\textit{f}}^{n}(p)}. 
\)   Notice that \begin{equation}\label{edre}  (a_{n},b_{n})=\tilde{f}_{n}-F_{n}\quad\text{and}\quad D(\tilde{f}_{n})_{x}-D(\tilde{f}_{n})_{0}    =  D(a_{n},b_{n})_{x}     \quad\text{for }x\in B_{n}(\delta_{n}).
\end{equation} 
For each $n\in\mathbb{Z}$, set \[\mu_{n}=\sup_{v\in E_{n} ^{s}}\frac{\Vert F_{n}v\Vert}{\Vert v\Vert} \quad\quad \text{ and } \quad\quad  \kappa_{n}=\sup_{v\in E_{n+1} ^{u}}\frac{\Vert F_{n}^{-1}v\Vert}{\Vert v\Vert}.\] 
It is clear that $\max\{\mu_{n},\kappa_{n} \}\leq  \lambda ,$ for all $n.$  
Set    \begin{equation*}
 \sigma_{n}  := \sigma_{n}(\delta_{n}) = \underset{x\in B_{n}(\delta_{n})}\sup\{\Vert D(a_{n})_{x}\Vert ,  \Vert D(b_{n})_{x}\Vert  \} . 
\end{equation*}  
 
  The following proposition   is shown in \cite{luisb}, Proposition 7.3.5, when   there exists   $\delta>0$ such that, considering $\delta_{n}=\delta$ for all $n$, $\sigma_{n}$ satisfies the second inequality in \eqref{deltamcrecimiento}  (notice that $\frac{\kappa^{-1} _{n}+\alpha\mu_{n}}{1+\alpha} >1$ for each $n$).  
 We have adapted that proof to obtain   a more general result, in which $\delta_{n}$ may vary with $n$ but satisfying the first condition in \eqref{deltamcrecimiento} (this fact means that $\delta_{n}$ must not decay very quickly as $n\rightarrow -\infty$). Furthermore, in our case, $\omega_{n}$, which will be defined above, could be very large (note that    $\kappa_{n}$ could be very large).    For \(\gamma\in ( \lambda^{2},1)\) and    $\tilde{\lambda}\in (\frac{1+\lambda}{2},1)$, set \[\omega_{n}  = \min\left\{  \frac{(\kappa_{n} ^{-1}-\mu _{n})\alpha}{(1+\alpha)^{2}} , \frac{(\gamma \kappa_{n} ^{-1}-\mu _{n})}{(1+\alpha)(1+\gamma)},\frac{2\lambda\tilde{\lambda}\kappa_{n}^{-1}-1- \lambda}{1+ \lambda}\right\}.\]  
 
 \begin{propo}\label{funciongrafico}   Suppose that for each $n\leq0$  we can  choose  the $\delta_{n}$'s such that 
\begin{equation}\label{deltamcrecimiento} 
    \frac{\kappa^{-1}_{n-1}+\alpha\mu_{n-1}}{1+\alpha} \delta_{n-1} \geq \delta_{n}   \quad\text{ and }\quad   
 \sigma_{n}   < \omega_{n}. 
\end{equation} 
Then, there exists a sequence of positive numbers $(\delta_{n})_{n\geq0}$ such that, for each $n\in\mathbb{Z}$, if    $\phi_{n}\in \Gamma^{u}_{n}(\alpha,\delta_{n}) $, we have that  $$\{\tilde{f}_{n}( \phi_{n}(w),w):w\in B_{n} ^{u}(\delta_{n})\} \cap B_{n+1}^{s}(\delta_{n+1}) \times B_{n+1}^{u}(\delta_{n+1})$$
is the $\mathcal{G}$ of an  application $\psi_{n+1}$ in $\Gamma^{u}_{n+1}(\alpha,\delta_{n+1}) $.  (see Figure \ref{funciongra}). \end{propo} 
\begin{figure}[ht] 
\begin{center}

\begin{tikzpicture}
\draw[black!7,fill=black!10, ultra thin] (-7.9,-0.4)rectangle (-4.1,3.4);
\draw[black, fill=black!50, thin] (-6,1.5) -- (-4.1,2.2) -- (-4.1,0.9) -- cycle;
\draw[black, fill=black!50, thin] (-6,1.5) -- (-7.9,2.2) -- (-7.9,0.9) -- cycle;

\draw[black!7,fill=black!10, ultra thin] (-2.4,-0.4)rectangle (1.4,3.4);
\draw[black, fill=black!50, thin] (-0.5,1.5) -- (1.4,2.2) -- (1.4,0.9) -- cycle;
\draw[black, fill=black!50, thin] (-0.5,1.5) -- (-2.4,2.2) -- (-2.4,0.9) -- cycle;

\draw[<->] (-6,-0.5) -- (-6,3.5);
\draw[<->] (-8.2,1.5) -- (-3.8,1.5); 
\draw (-6.2,4) node[below] {\quad{\small $E_{p}^{s}$}}; 
\draw (-3.7,1.7) node[below] {\quad{\small $E_{p}^{u}$}};
\draw (-6,-0.4) node[below] {\small $B^{s}_{n}(\delta_{n})\times B^{u}_{n}(\delta_{n})$};
\draw (-3.3,2.4) node[below] {\small $\tilde{f}_{n}$};
\draw[->, very thick] (-3.8,1.8) -- (-2.7,1.8);

\draw[<->] (-0.5,-0.5) -- (-0.5,3.5);
\draw[<->] (-2.7,1.5) -- (1.7,1.5);
\draw (-0.7,4) node[below] {\quad{\small $E_{q}^{s}$}};
\draw (2,1.7) node[below] {\small $E_{q}^{u}$};
\draw (-0.5,-0.4) node[below] {\small $B^{s}_{n+1}(\delta_{n+1})\times B^{u}_{n+1}(\delta_{n+1})$};

\draw[very thick] (-7.9,1.3) .. controls (-6.4,1) and (-5.7,2) .. (-4.1,1.65);

\draw (-5.1,2.45) node[below] {\small $\mathcal{G}(\phi_{n})$};
\draw[->] (-5.1,1.95) -- (-5.1,1.75);

\draw[very thick] (-2.70,1) .. controls (-1.7,1) and (0.3,2) .. (1.6,1.65);

 \draw (0.2,2.45) node[below] {\small $\mathcal{G}(\psi_{n+1})$};
 \draw[->] (0.1,1.95) -- (0.1,1.65);
\end{tikzpicture} 
\end{center}
\caption{$\mathcal{G}(\psi_{n+1})\subseteq\tilde{f}_{n}(\mathcal{G}( \phi_{n}))$. Shaded regions represent the unstable $\alpha$-cones.} \label{funciongra}
\end{figure}
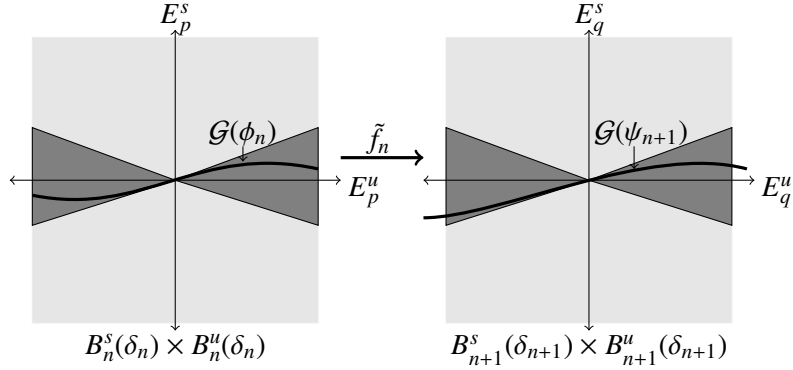
 
\begin{proof}  
Inductivelly, for each $n\geq0$ we can choose $\delta_{n}>0$ such that $\sigma_{n} <\omega_{n} $ and if    $\phi_{n-1}\in \Gamma^{u}_{n-1}(\alpha,\delta_{n-1}) $, then  \(\{\tilde{f}_{n-1}( \phi_{n-1}(w),w):w\in B_{n-1} ^{u}(\delta_{n-1})\} \cap B_{n}^{s}(\delta_{n}) \times B_{n}^{u}(\delta_{n})\)
is the $\mathcal{G}$ of an  application $\psi_{n}$ in $\Gamma^{u}_{n}(\alpha,\delta_{n}) $. 

Now, fix $n\in\mathbb{Z}$ and let   $\phi_{n}\in \Gamma^{u}_{n}(\alpha,\delta_{n}) $. For  $w\in B_{n}^{u}(\delta_{n})$, set 
\[  
r_{n}(w)=F_{n}w+b_{n}(\phi_{n}(w),w).
\]
If $w,z\in B_{n}^{u}(\delta_{n})$ we have \begin{equation}\label{rm} \Vert r_{n}(w)-r_{n}(z)\Vert \geq (\kappa_{n} ^{-1}  -\omega_{n}  (1+\alpha))\Vert w -z\Vert  
\end{equation}
and therefore $r_{n}$ is injective (notice that we have   $\kappa_{n} ^{-1}  -\omega_{n}  (1+\alpha )>0$). 
Furthermore, by  \eqref{deltamcrecimiento}  and choosing properly the $\delta_{n}$'s for $n> 0$,        we can obtain   \( B_{n}^{u}(\delta_{n+1})\subseteq r_{n}(B_{n}^{u}(\delta_{n}))\)  for each \(n\in\mathbb{Z}.
\)   
Consequently, we can  define the map  
$\psi_{n+1}:B_{n+1}^{u}(\delta_{n+1})\rightarrow E^{s} _{n+1},$ as    
\begin{equation}\label{fdephi}\psi_{n+1}(w)=F_{n}\phi_{n}(r_{n}^{-1}(w))+a_{n}(\phi_{n}(r_{n}^{-1}(w)),r_{n}^{-1}(w))\quad\text{for }w\in B_{n+1}^{u}(\delta_{n+1}).
\end{equation} 
Now, if  $x=r_{n}(w), y=r_{n}(z)\in B_{n+1}^{u}(\delta_{n+1})$, it follows from  \eqref{rm}  that
\begin{equation}\label{psim}\Vert \psi_{n+1}(x)-\psi_{n+1}(y)\Vert  \leq\frac{\alpha\mu_{n}  + \omega_{n}  (1+\alpha)}{\kappa_{n}^{-1}  -\omega_{n}  (1+\alpha)}\Vert  r_{n}(w)-r_{n}(z)\Vert \leq  \alpha\Vert  x-y\Vert
\end{equation}
(since $\omega_{n} \leq \frac{(\kappa_{n}^{-1}-\mu_{n})\alpha}{ (1+\alpha)^{2}} $, then   $\frac{\alpha\mu_{n}  + \omega_{n}  (1+\alpha)}{\kappa_{n} ^{-1}  -\omega_{n}  (1+\alpha)}\leq \alpha$). Thus,  
 $\psi_{n+1}$ is $\alpha$-Lipschitz. It is clear that  $\psi_{n+1}(0)=0$ and, since $\alpha< 1,$    from \eqref{psim} we have    $\psi_{n+1}(B_{n+1}^{u}(\delta_{n+1}))\subseteq B_{n+1}^{s}(\delta_{n+1}).$  Consequently, $\psi_{n+1}\in \Gamma^{u} _{n+1}(\alpha,\delta_{n+1}).$ 
On the other hand, if   $x=r_{n}(w)\in B_{n+1}^{u}(\delta_{n+1})$ we have 
\begin{equation}\label{psidef} (\psi_{n+1}(x),x)=(F_{n}\phi_{n}(w)+a_{n}(\phi_{n}w,w),F_{n}(w)+b_{n}(\phi_{n}w,w))= \tilde{f}_{n}(\phi_{n}w,w).
\end{equation}
Therefore, $\{\tilde{f}_{n}( \phi_{n}(w),w):w\in B_{n}^{u}(\delta_{n})\} \cap B_{n+1}^{s}(\delta_{n+1})\times B_{n+1}^{u}(\delta_{n+1})$ 
is the  $\mathcal{G}$ of $\psi_{n+1}$. This fact proves the proposition. 
\end{proof}
   
 \begin{obs}\label{classec2} Notice that   $a_{n}(0)=b_{n}(0)=D(a_{n})_{0}=D(b_{n})_{0}=0$ (see \eqref{edre}), consequently we always can choose   $\delta_{n}>0$ satisfying the second  condition in \eqref{deltamcrecimiento}.    If each $f_{n}$ is $C^{2}$ and the second derivative $p\to D ^{2}\textbf{\textit{f}}_{p} $, for $p\in\textbf{M}$, is    bounded,    then we can find a uniform  $\delta $ satisfying   \eqref{deltamcrecimiento} (see \cite{Jeo3}). On the other hand, in the random hyperbolic  dynamical system case, where all the maps $f_{i}$ are small perturbations of a fixed $C^{2}$ Anosov map $f$, we can also find a uniform $\delta$ satisfying  \eqref{deltamcrecimiento}. In our case, the $f_{i}$'s are not necessarily perturbations of a fixed map (see \cite{alb}, \cite{Jeo2}, \cite{Jeo3} for more detail). \end{obs}

From  Proposition \ref{funciongrafico} we have  the application \begin{align*}\textsf{G}:\Gamma^{u} (\alpha,(\delta_{n})_{n}) \rightarrow \Gamma^{u} (\alpha,(\delta_{n})_{n}),\quad
(\phi_{n})_{n\in\mathbb{Z}} \mapsto(\psi_{n-1})
_{n\in\mathbb{Z}},
\end{align*}  
where $\psi_{n}$ is given in  \eqref{fdephi}, is well defined.  We can prove that  
  $\textsf{G}$ is a contraction, with contraction constant $\gamma$.
Since  $\Gamma^{u}(\alpha,(\delta_{n})_{n})$ is a    complete   metric space, by the Banach fixed-point Theorem we have there exists an unique  $\phi^{\star}\in \Gamma^{u} (\alpha,(\delta_{n})_{n})$ such that $\textsf{G}(\phi^{\star})=\phi^{\star}.$  
In other words,   for each  $n\in\mathbb{Z}$, there exists an unique  $\phi_{n}^{\star} \in \Gamma^{u} _{n}(\alpha, \delta_{n})$ such that  $\tilde{f}_{n}(\phi_{n}^{\star} w,w)=(\phi_{n+1}^{\star}r_{n}(w),r_{n}(w)),$  for all $w\in B_{n}^{u}(\delta_{n})$  (see \eqref{psidef}).
Consequently,  if \begin{equation}\label{manifold}V_{n}(\delta_{n})=\{(\phi_{n}^{\star} w,w):w\in B_{n}^{u}(\delta_{n})\},
\end{equation} we have  $V_{n+1}(\delta_{n+1})\subseteq \tilde{f}_{n}(V_{n}(\delta_{n}))$ (remember that  $B_{n+1}^{u}(\delta_{n+1})\subseteq r_{n}(B_{n}^{u}(\delta_{n}))$).

 \begin{obs}\label{difernciabilidademanifolds}   The sets $V_{n}(\delta_{n})$ are  submanifolds of $M_{n}$, because they are
  graphs of applications $\alpha$-Lipschitz. The differentiable structure of $V_{n}(\delta_{n})$ is obtained from the differentiability of $\phi_{n}^{\star}$, which   can be shown similarly as in \cite{luisb}, p. 201. Furthermore,   we have  $T_{0}V_{n}(\delta_{n})=E_{\textbf{\textit{f}}^{n}(p)} ^{u}$. \end{obs}

Since $\omega_{n}\leq\frac{2\lambda\tilde{\lambda}\kappa_{n}^{-1}-1- \lambda}{1+ \lambda}$, we can prove that     $$\tau_{n}:=\frac{1+\alpha}{\kappa_{n} ^{-1}-\omega_{n}(1+\alpha)}=\frac{1+ \lambda}{ 2 \lambda \kappa_{n}^{-1} - \omega_{n}(1+ \lambda)}<\tilde{\lambda}\quad \text{ for each }n\in\mathbb{Z}.$$  We will see   the contraction of the  submanifolds by \textbf{\textit{f}} can be controlled by $\tau_{n}$ and also depends on the  angles between the stable and unstable.   Notice that, if the angles $\theta_{n}$ decay as $n\rightarrow \pm \infty$,  the vectors in  $B^{s}_{n}(\delta_n)$  and in $B^{u}_{n}(\delta_n)$ are
ever closer. 

  \medskip
  
  Fix $\zeta \in (0,1-\lambda) $   and let $\theta_{n}$ be as in \eqref{vainadelosangulos}. In \cite{Jeo2}, Proposition 3.7,  we prove that there exists a Riemannian metric $\langle \cdot,\cdot\rangle_{\ast}$ on \textbf{M} such that $(\textbf{M},\langle \cdot,\cdot\rangle_{\ast},\textbf{\textit{f}})$ is a  strictly Anosov family with constant $ \lambda^{\prime}=\lambda+\zeta$ and the   stable and unstable subspaces are orthogonal.  Furthermore, we have that \begin{equation}\label{cosadeequiv} \Delta_{n} \Vert v\Vert_{\ast}\leq  \Vert v\Vert \leq 2\Vert v\Vert _{\ast},\quad\text{for }v\in TM_{n},
\end{equation} 
where $\Vert \cdot\Vert_{\ast}$ is the norm induced by $\langle \cdot,\cdot\rangle_{\ast}$ and $\Delta_{n}=\left(\frac{1}{1-\cos(\theta_{n})}(\frac{\lambda+\zeta}{\zeta})^{2}\right)^{-1}.$   

\medskip

 Hence, if  $w\in B_{n}^{u}(\delta_{n})$ for $n\in\mathbb{Z}$, by \eqref{rm},  \eqref{psidef} and \eqref{cosadeequiv} 
   we have
\begin{align*} \Vert \tilde{f}_{n}(\phi_{n}^{\star} w,w)\Vert &\geq \Delta_{n+1}\Vert \tilde{f}_{n}(\phi_{n}^{\star} w,w)\Vert_{\ast}\geq \Delta_{n+1} \Vert r_{n}(w)\Vert_{\ast}  \\
&\geq \frac{\Delta_{n+1}}{2}\frac{\kappa_{n}^{-1}-\omega_{n}(1+\alpha)}{1+\alpha}\Vert (\phi_{n}^{\star} w,w)\Vert.
\end{align*}
Consequently, since  $(\tilde{f}_{n})^{-1}(V_{n+1}(\delta_{n+1}))\subseteq  V_{n}(\delta_{n})$, for every $n\in\mathbb{Z}$, we have  that
\[\Vert (\tilde{f}_{n})^{-1}(\phi_{n+1}^{\star} w,w)\Vert\leq \frac{2}{\Delta_{n+1}}\tau_{n}\Vert  (\phi_{n+1}^{\star} w,w)\Vert ,\quad\text{ for } w\in B^{u}_{n+1}(\delta_{n+1}).\]
Inductively we can prove for $k \geq  0$, if 
  $w\in B^{u}_{n+1}(\delta_{n+1})$ then    \begin{equation}\label{desigualdadeexponenciales}\Vert (\tilde{f}_{n-k})^{-1}\circ\dots\circ(\tilde{f}_{n})^{-1}(\phi_{n+1}^{\star} w,w)\Vert\leq \frac{2}{\Delta_{n+1}}\tau_{n-k} \cdots\tau_{n}\Vert  (\phi_{n+1}^{\star} w,w)\Vert .
\end{equation} 

\begin{teo}\label{primerpropovariedade}  Fix   $p\in M_{0}$. Suppose   the  Anosov family $(\textbf{M},\langle\cdot,\cdot\rangle, \textbf{\textit{f}})$ admits a sequence of positive numbers $\delta=(\delta_{n})_{n\in\mathbb{Z}}$ as in Proposition \ref{funciongrafico}. Thus,  there exists a  two-sided sequence      $\{\mathcal{W}^{u}( \textbf{\textit{f}}_{0} ^{\, n}(p),\delta):n\in\mathbb{Z}\}$, where   $\mathcal{W}^{u}( \textbf{\textit{f}}_{0} ^{\, n}(p),\delta)$ is a differentiable submanifold of $M_{n}$ with size $2\delta_n$, such that  for   $n\in\mathbb{Z}$:
\begin{enumerate}[\upshape (i)]
\item $\textbf{\textit{f}}^{\, n}_{0}(p)\in \mathcal{W}^{u}( \textbf{\textit{f}}_{0} ^{\, n}(p),\delta)$ and  
  $T_{\textbf{\textit{f}}^{\, n}_{0}(p)}\mathcal{W}^{u}(\textbf{\textit{f}}^{\, n}_{0}(p),\delta)=E^{u}_{\textbf{\textit{f}}^{\, n}_{0}(p)}$,  
\item $f_{n-1} ^{-1}(\mathcal{W}^{u}(  \textbf{\textit{f}}_{0} ^{\, n}(p),\delta)) \subseteq \mathcal{W}^{u}(  \textbf{\textit{f}}_{0} ^{\, n-1}(p),\delta)$, and furthermore
\item if $q\in \mathcal{W}^{u}( p_{n+1},\delta )$, where $p_{n}=\textbf{\textit{f}}_{0}^{\, n}(p)$, and $k\geq0$ we have  
\begin{equation}\label{desig1}  d(\textbf{\textit{f}}_{n+1}^{-(k+1)}(q),\textbf{\textit{f}}_{n+1}^{-(k+1)}(p_{n+1}))\leq \frac{2}{\Delta_{n+1}}\tau_{n-k} \cdots\tau_{n}d(q,p_{n+1}). 
\end{equation}    
\end{enumerate}
\end{teo}
\begin{proof} Let $V_{n}(\delta_{n})$ be as in \eqref{manifold}  and take   $\mathcal{W} ^{u}(\textbf{\textit{f}}_{0} ^{\, n}(p),\delta) =\text{exp}_{\textbf{\textit{f}}^{n}(p)}(V_{n}(\delta_{n}))$ for each $n$.   The statements (i) and (ii) of the theorem are clear.  

For (iii);   if  $q\in \mathcal{W}^{u}( p_{n},\delta)$, for each $k\geq 1$ there exists a unique $v_{n-k+1}\in T_{\textbf{\textit{f}}_{n+1}^{-k}(p)}\textbf{M}$ such that  $\text{exp}_{\textbf{\textit{f}}_{n+1}^{-k}(p)}(v_{n-k+1})=\textbf{\textit{f}}_{n+1}^{-k}(q)$ and  $\Vert v_{n-k+1}\Vert =d(\textbf{\textit{f}}_{n+1}^{-k}(p),\textbf{\textit{f}}_{n+1}^{-k}(q)).$ 
By \eqref{desigualdadeexponenciales} and since       $\mathcal{W}^{u}( \textbf{\textit{f}}_{0} ^{\, n}(p),\delta)$ is invariant by   $\textbf{\textit{f}}$ we have \eqref{desig1}. 
\end{proof}
  Theorem \ref{primerpropovariedade} is a more generalized  version of the  Hadamard-Perron Theorem    adapted  to Anosov families for the unstable case, since the angles between the stable and unstable subspace could be arbitrarily small and, furthermore, the $\delta_{n}$'s   satisfying the condition \eqref{deltamcrecimiento} are not necessarily uniform.

\medskip

Analogously we can obtain a more generalized  version of the Hadamard-Perron Theorem adapted to Anosov families for the stable case. Indeed, consider the sequence $(\epsilon_{n})_{n\in\mathbb{Z}}$  of positive numbers small enough such that  \begin{equation*} \hat{f}_{n} =\text{exp}^{-1} _{f^{n}(p)}\circ f_{n}^{-1}\circ \text{exp}_{f^{n+1}(p)}:B_{n+1}(\epsilon_{n+1})  \rightarrow T_{f^{n}(p)}\textbf{M}
\end{equation*}  is well-defined.   
For $ (v,w)\in B_{n+1}(\epsilon_{n+1}),$ set
\begin{equation*}\hat{f}_{n}(v,w)=(c_{n}(v,w)+G_{n}(v) ,d_{n}(v,w) +G_{n}(w)),  
\end{equation*}
 where  \(
G_{n} =D(f_{n}^{-1})_{\textbf{\textit{f}}_{0}^{n+1}(p)}.\) 
Notice that, for each $n\in\mathbb{Z}$,    $$ \sup_{v\in E_{n} ^{s}}\frac{\Vert G_{n}^{-1}v\Vert}{\Vert v\Vert}=\sup_{v\in E_{n} ^{s}}\frac{\Vert F_{n} v\Vert}{\Vert v\Vert}=\mu_{n} \quad\text{ and }  \quad \sup_{v\in E_{n+1} ^{u}}\frac{\Vert G_{n}v\Vert}{\Vert v\Vert}=\sup_{v\in E_{n+1} ^{u}}\frac{\Vert F^{-1}_{n}v\Vert}{\Vert v\Vert}=\kappa_{n}.$$ 
Set  \[\rho_{n}:=\rho_{n}(\epsilon_{n+1}) =
  \underset{x\in B_{n+1}(\epsilon_{n+1})}\sup  \{\Vert D( c_{n})_{x}\Vert   ,\Vert D( d_{n})_{x} \Vert  \}.\] 
Suppose that  there exists $(\epsilon_{n})_{n\geq0}$ such that  \(\epsilon_{n-1}\leq \frac{\mu^{-1} _{n}+\alpha\kappa_{n}}{1+\alpha} \epsilon_{n}\)  for $n\geq 0$ and 
\begin{equation*}\label{deltamcrecimiento2}  
\rho_{n} <\varpi_{n}:=\min\left\{\frac{(\mu_{n}^{-1}-\kappa_{n})\alpha}{(1+\alpha)^{2}} ,\frac{(\gamma\mu_{n}^{-1}-\kappa_{n})}{(1+\alpha)(1+\gamma)},\frac{2\lambda\tilde{\lambda}\mu_{n}^{-1}-1- \lambda}{1+ \lambda} \right\}.
\end{equation*}

  There exists a sequence $(\epsilon_{n})_{n\leq -1}$ such that, considering $\epsilon=(\epsilon_{n})_{n\in\mathbb{Z}}$,   we have:

\begin{teo}\label{primerpropovariedade2} There exists a two-sided sequence $\{\mathcal{W}^{s}( \textbf{\textit{f}}_{0} ^{\, n}(p),\epsilon):n\in\mathbb{Z}\}$   of differentiable  submanifold of $M_{n}$ of size $2\epsilon_{n}$  for each  $n\in\mathbb{Z}$,     such that
\begin{enumerate}[\upshape (i)]
\item $\textbf{\textit{f}}^{\, n}_{0}(p)\in \mathcal{W}^{s}( \textbf{\textit{f}}_{0} ^{\, n}(p),\epsilon)$ and  $T_{\textbf{\textit{f}}^{\, n}_{0}(p)}\mathcal{W}^{s}( \textbf{\textit{f}}^{\, n}_{0}(p),\epsilon)=E^{s}_{\textbf{\textit{f}}^{\, n}_{0}(p)}$,  
\item $f_{n} (\mathcal{W}^{s}(  \textbf{\textit{f}}_{0} ^{\, n}(p),\epsilon))\subseteq \mathcal{W}^{s}(  \textbf{\textit{f}}_{0} ^{\, n+1}(p),\epsilon)$, and furthermore
\item if  $q\in \mathcal{W}^{s}( p_{n},\epsilon)$, where $p_{n}=\textbf{\textit{f}}_{0}^{\, n}(p)$, and $k\geq1$ we have  
\begin{equation}\label{desig121} d(\textbf{\textit{f}}_{n}^{k}(q),\textbf{\textit{f}}_{n}^{k}(p_{n}))\leq \frac{2}{\Delta_{n}}\varsigma_{n+k}\cdots\varsigma_{n+1}d(q,p_{n}),
\end{equation}
where $\varsigma_{k}=\frac{1+\alpha}{\mu_{k} ^{-1}-\varpi_{k}(1+\alpha)}$.  
\end{enumerate}
\end{teo}

 As in \cite{luisb}, we will call the manifold $\mathcal{W}^{u}(  p,\delta )$ as \textit{admissible} $(u,\alpha,\delta)$-\textit{manifold at} $p$ and $\mathcal{W}^{s}(  p,\epsilon )$ as \textit{admissible} $(s,\alpha,\epsilon)$-\textit{manifold at} $p$.  These manifolds do not necessarily coincide with the sets given  in Definition \ref{conjuntosestaviesfam},     since    $\Delta_{k}$ could   decrease quickly
 when $k\rightarrow \pm\infty.$

\medskip

The first inequality \eqref{deltamcrecimiento} means that the radius $\delta_{n}$  of the balls  $B_{n}^{u}(\delta_{n})$  must not  decrease  very fast  as $n\rightarrow -\infty$. This condition is sufficient for the invariance of the admissible  manifolds obtained in   Theorem \ref{primerpropovariedade}  by  \textbf{\textit{f}} (see \eqref{rm}). Remember   we have considered exponential charts to work on the ``ambient Euclidian'' and the $\delta_{n}$'s  depend  on both   \textbf{\textit{f}} and $r(\textbf{\textit{f}}^{n}(p))$. This fact is of great importance to the construction of unstable  (stable) manifolds, because the expansions (contractions) of each manifold could be caused by the   geometry of  each component but not because of the family (see Example \ref{contraexemplo}).

\section{Local Stable  and Unstable  Manifolds for Anosov Families}
In the previous section we obtained admissible manifolds for Anosov families whose expansion  or contraction  are well controlled. In this section we will give certain conditions with which  the stable and unstable sets   (see Definition \ref{conjuntosestaviesfam})     coincide  with the admissible manifolds (see Lemmas  \ref{limitedosangulos} and \ref{lema2deigualdade}). Finally, in   Theorems   \ref{variedadeinstave}     and \ref{variedadeestavel} we show the main   objective of this work, the   unstable and stable manifold Theorems  for Anosov families.

\medskip

 We had talked about the importance of maintaining the metrics established, because
the notion of the  Anosov  family depends on the Riemannian metrics on each $M_{n}$ (see \cite{alb}, \cite{Jeo2}). Changing the metrics on each component  we could get stable (unstable) sets  
very different  (see Example  \ref{contraexemplo}). However, we have: 

 \begin{propo}\label{porpomantenimientovariedades}  Let $\langle\cdot,\cdot\rangle$ and $\langle\cdot,\cdot\rangle^{\prime}$ be uniformly equivalent  Riemannian metrics   on \textbf{M}.  Fix $p\in M_{0}$. There exist   sequences of   positive  numbers   $ \varepsilon=(\varepsilon_{i})_{i\in\mathbb{Z}}$,  $\varepsilon^{\prime}= (\varepsilon_{i} ^{\prime})_{i\in\mathbb{Z}}$ and $ \tilde{\varepsilon}=(\tilde{\varepsilon}_{i}  )_{i\in\mathbb{Z}}$ such that, for   $r=u,s,$ 
 $$\mathcal{N}^{r}(p,\varepsilon,\langle\cdot,\cdot\rangle)\subseteq \mathcal{N}^{r}(p,\varepsilon   ^{\prime}, \langle\cdot,\cdot\rangle^{\prime})\subseteq \mathcal{N}^{r}( p,\tilde{\varepsilon}   , \langle\cdot,\cdot\rangle ).$$    
 \end{propo}
 \begin{proof} We will show only the stable case,   since the unstable case  is analogous.  Consider  $\textbf{N}=\textbf{M}$ with the Riemannian metric  $\langle\cdot,\cdot\rangle^{\prime}$,  that is, $N_{i}=M_{i}$ with the Riemannian metric  $\langle\cdot,\cdot\rangle_{i}^{\prime}:=\langle\cdot,\cdot\rangle^{\prime}|_{M_{i}}$ for each $i\in\mathbb{Z}.$ Let  $I_{i}:(M_{i},\langle\cdot,\cdot\rangle)\rightarrow (N_{i},\langle\cdot,\cdot\rangle^{\prime})$ be the identity.      For each $n\geq 0$ we  can find a    $\varepsilon_{n}>0$ small enough such that 
     \[\text{diam} [B(\textbf{\textit{f}}_{0} ^{n}(p),\varepsilon_{n} ,\langle\cdot,\cdot\rangle)]<r(\textbf{\textit{f}}_{0} ^{n}(p))\text{ and }\text{diam} [I_{n}(B( \textbf{\textit{f}}_{0} ^{n}(p),\varepsilon_{n}   ,\langle\cdot,\cdot\rangle))]<r^{\prime}(\textbf{\textit{f}}_{0} ^{n}(p))\] 
 (we use the notations  $B( \textbf{\textit{f}}_{i} ^{n}(p),\varepsilon_{n},\langle\cdot,\cdot\rangle)$ for the ball in  $ (M_{n},\langle\cdot,\cdot\rangle)$  and   $r^{\prime}(\textbf{\textit{f}}_{0} ^{n}(p))$   for the injectivity radius of the exponential map at $\textbf{\textit{f}}_{0} ^{n}(p)$ considering the metric $\langle\cdot,\cdot\rangle^{\prime}$ on \textbf{M}).    
For each $n\geq0$, take  $$\varepsilon_{n} ^{\prime}=\frac{1}{2}\text{diam} [I_{n}(B(  \textbf{\textit{f}}_{0} ^{n}(p),\varepsilon_{n} ,\langle\cdot,\cdot\rangle))].$$
   Fix  $q\in \mathcal{N}^{s}( p,\varepsilon,\langle\cdot ,\cdot\rangle).$ Thus $I_{n}(\textbf{\textit{f}}_{0} ^{n}(q))\in B( \textbf{\textit{f}}_{0} ^{n}(p),\varepsilon^{\prime}_{n} ,\langle\cdot,\cdot\rangle^{\prime})$ for all $n\geq0.$  
Let  $v\in T_{p}M_{0}$ be such that $\text{exp}_{p}(v)=q.$ Since $\langle\cdot,\cdot\rangle$ and $\langle\cdot,\cdot\rangle^{\prime}$ are uniformly equivalent, there exist positive numbers $k,K$ such that $k\Vert v\Vert ^{\prime}\leq \Vert v\Vert \leq K\Vert v\Vert ^{\prime}$, for all $v\in T_{\textbf{\textit{f}}^{n} (p)}M_{n}$, $n\geq 0$, where $\Vert \cdot\Vert$ and $\Vert\cdot\Vert^{\prime}$ are the norms induced by $\langle\cdot,\cdot\rangle$ and $\langle\cdot,\cdot\rangle^{\prime}$, respectively.   Thus,  \[\frac{1}{n}\log d(\textbf{\textit{f}}^{n}(p),\textbf{\textit{f}}^{n}(q)) =\frac{1}{n}\log\Vert  \tilde{f}_{i+n-1}  \circ\dots\circ\tilde{f}_{i} (v)\Vert \geq \frac{1}{n}\log k\Vert  \tilde{f}_{i+n-1}  \circ\dots\circ \tilde{f}_{i}  (v)\Vert^{\prime}.
\]
Therefore, $q\in \mathcal{N}^{s}( p,\varepsilon^{\prime},\langle\cdot,\cdot\rangle)$. Thus, $ \mathcal{N}^{s}( p, \varepsilon ,\langle\cdot,\cdot\rangle)\subseteq \mathcal{N}^{s}( p ,\varepsilon^{\prime},\langle\cdot,\cdot\rangle^{\prime})$.
Analogously we can prove  the existence of the sequence $\tilde{\varepsilon}=(\tilde{\varepsilon}_{i})_{i\in\mathbb{Z}}$ such that     $$ \mathcal{N}^{s}( p,\varepsilon^{\prime},\langle\cdot,\cdot\rangle^{\prime})\subseteq \mathcal{N}^{s}( p ,\tilde{\varepsilon},\langle\cdot,\cdot\rangle),$$ 
which proves   the proposition.
 \end{proof}

 In the next lemma we show  $\mathcal{W}^{u}(p,\delta)\subseteq \mathcal{N}^{u}(p,\delta )$ and $\mathcal{W}^{s}(p,\epsilon)\subseteq \mathcal{N}^{s}(p,\epsilon)$.   
 In  Lemma \ref{lema2deigualdade} we will give a condition for the reverse inclusion.

 \begin{lem}\label{limitedosangulos} For each $p\in  M_{0}$ we have 
 \[\mathcal{W}^{u}(p,\delta)\subseteq \mathcal{N}^{u}(p,\delta)\quad\text{ and }\quad  \mathcal{W}^{s}(p,\epsilon)\subseteq \mathcal{N}^{s}( p,\epsilon).\]
\end{lem}
\begin{proof} 
We will prove $\mathcal{W}^{u}(p,\delta)\subseteq \mathcal{N}^{u}(p,\delta)$.  Take $q\in \mathcal{W}^{u}(  p,\delta)$. By Theorem  \ref{primerpropovariedade},    we have $\textbf{\textit{f}}_{0}^{-n}(q)\in B(\textbf{\textit{f}}_{0}^{-n}(p),\delta_{-n})$ and  \[ d(\textbf{\textit{f}}_{0}^{-n}(q),\textbf{\textit{f}}_{0}^{-n}(p))\leq\frac{2}{\Delta_{0}}  \tau_{-n}\cdots\tau_{-1}d(p,q) \quad\text{ for each }n\geq1.\] 
 Since $\tau_{k}<\tilde{\lambda}<1 $,   we have     \(\underset{n\rightarrow \infty}\limsup \frac{1}{n}\log d(\textbf{\textit{f}}_{0}^{-n}(q),\textbf{\textit{f}}_{0}^{-n}(p))\leq \log \tilde{\lambda}<0.\)  
Consequently,  $q\in \mathcal{N} ^{s}( p,\delta)$. 
\end{proof} 

 
  

\begin{lem}\label{lema2deigualdade}   Set \begin{align*} \Omega &=   \underset{n\rightarrow -\infty}\limsup\,\theta_{n}\quad\text{ }\quad \tilde{\Omega}=\liminf_{n\rightarrow\infty} \frac{1}{n}\log\frac{\Delta_{-n}}{2} \varsigma_{0}^{-1}\cdots\varsigma_{-n+1}^{-1}  \\
\Theta &= \underset{n\rightarrow  \infty}\limsup\,\theta_{n} \quad\text{ }\quad\tilde{\Theta}=\liminf_{n\rightarrow\infty} \frac{1}{n}\log \frac{\Delta_{n}}{2}  \tau_{0}^{-1} \cdots\tau_{n-1}^{-1}
.\end{align*}  Thus
\begin{enumerate}[\upshape (i)]
\item Assume that    we can choose the $\delta_{n}$'s such that $\delta_{n}\leq \epsilon_{n}$ for each   $n\leq 0$. If $   \Omega>0$ and $\tilde{\Omega}\geq 0,$ then     there exists a sequence of positive numbers  $ \delta^{\prime}=(\delta_{n} ^{\prime})_{n\in \mathbb{Z}}$ such that $\mathcal{N}^{u}( p,\delta ^{\prime}) \subseteq  \mathcal{W}^{u}(p,\delta^{\prime} )$.
\item Assume that     we can choose the $\epsilon_{n}$'s such that $\epsilon_{n}\leq \delta_{n}$ for each   $n\geq 0$. If $ \Theta>0$ and $\tilde{\Theta}\geq 0$, then   there exists a sequence of positive numbers  $ \epsilon^{\prime}  =(\epsilon_{n} ^{\prime})_{n\in \mathbb{Z}}$ such that $\mathcal{N}^{s}( p,\epsilon^{\prime} ) \subseteq   \mathcal{W}^{s}( p,\epsilon^{\prime} )$.        
\end{enumerate} 
\end{lem}
\begin{proof}We will prove (i). Fix $\nu\in (0,\Omega)$. Let $(n_{i})_{i\in\mathbb{N}}$ be a sequence of natural numbers, with $0=n_{0}<n_{1}<\cdots <n_{m}<\cdots$, and  $\theta_{n_{i}}\geq \nu$ for each $i\geq0.$ Since $\delta_{n}\leq \epsilon_{n}$ and $   \nu>0$, we can choose   $\delta_{n}^{\prime} \leq \delta_{n}/3$ small enough such that     $\delta^{\prime}=(\delta_{n}^{\prime})_{n\in \mathbb{Z}}$  satisfies  \eqref{deltamcrecimiento}     and \[B(\textbf{\textit{f}}_{0}^{-n_{i}}(p),\delta_{-n_{i}}^{\prime})\subseteq  \mathcal{W}^{s}(  \textbf{\textit{f}}_{0}^{-n_{i}}(p),\epsilon )\times \mathcal{W}^{u}(\textbf{\textit{f}}_{0}^{-n_{i}}(p),\delta ) ,\quad\text{ for each }i\geq0.\]
   It follows from Theorem \ref{primerpropovariedade} there exists a family \(\{\mathcal{W}^{u}(\textbf{\textit{f}}_{0}^{n}(p),\delta^{\prime} ):n\in\mathbb{Z}\}\) of admissible $(u,\alpha, \delta^{\prime})$-manifold.  Next, we prove that $\mathcal{N} ^{u}( p,\delta^{\prime} )\subseteq \mathcal{W} ^{u}(p,\delta^{\prime} )$. 
    Indeed, suppose there exists  $q\in \mathcal{N}^{u}( p,\delta^{\prime} )\setminus  \mathcal{W} ^{u}(p,\delta^{\prime} )$.  Since $\textbf{\textit{f}}_{0} ^{-n_{i}}(q)\in B  (\textbf{\textit{f}}_{0} ^{-n_{i}}(p),\delta_{-n_{i}}^{\prime})$, we have     \[(\tilde{f}_{-n_{i}})^{-1}\circ\dots\circ(\tilde{f}_{-1})^{-1} (\text{exp}_{p} ^{-1}(q) )=(x_{-n_{i}},y_{-n_{i}})\in \mathcal{W}^{s}(  \textbf{\textit{f}}_{0} ^{-n_{i}}(p),\epsilon )\times \mathcal{W}^{u}(\textbf{\textit{f}}_{0} ^{-n_{i}}(p),\delta),  \] 
 for all $i\geq0,$ where $x_{-n_{i}}\in \mathcal{W}^{s}( \textbf{\textit{f}}_{0} ^{-n_{i}}(p)  ,\epsilon)\setminus\{0\}$ and $y_{-n_{i}} \in \mathcal{W}^{u}(  \textbf{\textit{f}}_{0} ^{-n_{i}}(p) ,\delta )$. We can obtain from \eqref{desig1}  and \eqref{desig121} that   \begin{align*}\Vert (x_{-n_{i}},& y_{-n_{i}})\Vert    \geq  \Vert x_{-n_{i}}\Vert -  \Vert y_{-n_{i}}\Vert  \geq \frac{\Delta_{-n_{i}}}{2}\varsigma_{0}^{-1}\cdots\varsigma_{-n_{i}+1}^{-1}  \Vert x_{0}\Vert  -\frac{2}{\Delta_{0}} \tilde{\lambda}^{n_{i}} \Vert y_{0}\Vert. 
\end{align*}
We have  $\frac{2}{\Delta_{0}} \tilde{\lambda}^{n_{i}} \Vert y_{0}\Vert \rightarrow 0$ as $i\rightarrow +\infty.$  Consequently, 
\begin{equation*}\limsup_{i\rightarrow\infty} \frac{1}{n_{i}}\log \Vert (x_{-n_{i}},  y_{-n_{i}})\Vert   \geq \liminf_{n\rightarrow\infty} \frac{1}{n}\log\frac{\Delta_{-n}}{2} \varsigma_{0}^{-1}\cdots\varsigma_{-n+1}^{-1} \geq 0.
\end{equation*}
 This fact contradicts that $q\in   \mathcal{N}^{s}(p,\delta^{\prime})$. 
\end{proof}

 From now on we will assume that for each $p\in M_{0}$ we can choose the sequences \((\delta_{n})_{n\in\mathbb{Z}}\) and \((\epsilon_{n})_{n\in\mathbb{Z}}\) as in Theorems \ref{primerpropovariedade} and \ref{primerpropovariedade2},  such that 
 \begin{equation}\label{cccc}
\delta_{-n}\leq\epsilon_{-n} \text{ and }\epsilon_{n}\leq\delta_n \text{ for all  }n>0, \, \, 0\leq \min\{\tilde{\Omega}, \tilde{\Theta}\}\,\text{ and }\, 0<\min\{\Omega,\Theta\} .
\end{equation} 

\begin{obs}It is not difficult to prove that  if \textbf{\textit{f}} satisfies the property of the angles then \(0\leq \min\{\tilde{\Omega}, \tilde{\Theta}\}\)  and \( 0<\min\{\Omega,\Theta\}.\)\end{obs}

We have by Lemmas \ref{limitedosangulos} and \ref{lema2deigualdade} that there exist two  sequences   of positive numbers $\delta^{\prime}=(\delta_{n}^{\prime} )_{n\in\mathbb{Z}} $ and $\epsilon^{\prime}=(\epsilon_{n}^{\prime} )_{n\in\mathbb{Z}} $ such that  \(\{\mathcal{W}^{u}(\textbf{\textit{f}}_{0}^{n}(p),\delta^{\prime} ):n\in\mathbb{Z}\}\) and \(\{\mathcal{W}^{s}(\textbf{\textit{f}}_{0}^{n}(p),\epsilon^{\prime}):n\in\mathbb{Z}\}\) are two-sided sequences of  admissible   manifolds, and furthermore $\mathcal{N}^{u}(p,\delta^{\prime} ) =\mathcal{W} ^{u}(p,\delta^{\prime} )$ and $\mathcal{N}^{s}(p,\epsilon^{\prime} ) =  \mathcal{W}^{s}(p,\epsilon^{\prime})$.  

\medskip

 Next we will prove that $\mathcal{N} ^{u}(p,\delta^{\prime})$ and $ \mathcal{N}^{s}(p,\epsilon^{\prime})$    depend  continuously on   $ p $.  
  
\begin{lem}\label{continuidadsub}  Let $(p_{m})_{m\in \mathbb{N}}$   be a  sequence in $ M_{0}$ converging to $p\in M_{0}$ when $m\rightarrow \infty$. If  $q_{m}\in   \mathcal{N}^{r}(p_{m},\eta)$ converges to $q\in B(p,\eta_{0})$ as $m\rightarrow \infty$,  then $q\in \mathcal{N}^{r}(p,\eta)$, for $r=s,u$. 
\end{lem}
\begin{proof}We will prove only  the stable case.  Set $$ \omega=\underset{p_{m},m \geq0}\sup \{\underset{n\rightarrow \infty}\limsup \frac{1}{n}\log d(\textbf{\textit{f}}_{0}^{n}(q^{\prime}),\textbf{\textit{f}}_{0}^{n}(p_{m})):q^{\prime}\in N^{s}(p_{m},\eta)\}.$$ By compactness of $M_{0}$, we have $ \omega<0.$  Fix $\beta\in(0,\exp( \omega))$. For each $n\in \mathbb{N},$ take  \(m_{n}\in \mathbb{N}\), with  \(m_{1}< \cdots <m_{n}<\cdots\),  such that 
\(d(\textbf{\textit{f}}_{0}^{n}(p_{m_{n}}),\textbf{\textit{f}}_{0}^{n}(p))< \beta^{n} \) and \( d(\textbf{\textit{f}}_{0}^{n}(q_{m_{n}}),\textbf{\textit{f}}_{0}^{n}(q))< \beta^{n} \).  
For every $n$ we have 
\begin{align*} \frac{1}{n}&\log (d(\textbf{\textit{f}}_{0}^{n}(q ),\textbf{\textit{f}}_{0}^{n}(p )))\\
&\leq \frac{1}{n}\log[d(\textbf{\textit{f}}_{0}^{n}(q ),\textbf{\textit{f}}_{0}^{n}(q_{m_n}))+d(\textbf{\textit{f}}_{0}^{n}(q_{m_{n}}),\textbf{\textit{f}}_{0}^{n}(p_{m_{n}}))+d(\textbf{\textit{f}}_{0}^{n}(p_{m_{n}}),\textbf{\textit{f}}_{0}^{n}(p))]. 
\end{align*}
Since $\underset{n\rightarrow \infty}\limsup\frac{1}{n}\log(a_{n}+b_{n})=\max\{\underset{n\rightarrow \infty}\limsup\frac{1}{n}\log(a_{n}),\underset{n\rightarrow \infty}\limsup\frac{1}{n}\log(b_{n})\}$ for any sequence of positive numbers $a_{n}$ and $b_{n}$, we have \[\underset{n\rightarrow \infty}\limsup\frac{1}{n} \log (d(\textbf{\textit{f}}_{0}^{n}(q ),\textbf{\textit{f}}_{0}^{n}(p )))\leq \max\{\log(\beta),\omega\} =\omega.\]
Consequently, 
$q\in \mathcal{N}^{s}(p,\eta)$.
\end{proof}

Finally, by   Theorems   \ref{primerpropovariedade} and \ref{primerpropovariedade2} and    Lemmas    \ref{limitedosangulos}-\ref{continuidadsub}, we obtain the following  local unstable and stable manifold theorems for Anosov families: 

\begin{teo}\label{variedadeinstave}  If     $(\textbf{M},\langle\cdot,\cdot\rangle, \textbf{\textit{f}}\,)$ admits a sequence of positive numbers $(\delta_{n})_{n\in\mathbb{Z}}$ satisfying  \eqref{cccc}  for   each $p\in M_{0}$, then there exists a sequence of positive numbers $\eta= (\eta_{n})_{n\in\mathbb{Z}},$ such that    $\mathcal{N}^{u}( \textbf{\textit{f}}_{0} ^{\, n}(p) ,\eta)$  is a differentiable submanifold of $M_{n}$ with:
 \begin{enumerate}[\upshape (i)]
\item   $T_{\textbf{\textit{f}}^{n}(p)}\mathcal{N}^{u}( \textbf{\textit{f}}^{\, n}_{0}(p),\eta)=E_{\textbf{\textit{f}}^{n}(p)} ^{u}$;
\item $f_{n-1} ^{-1}(\mathcal{N}^{u}(  \textbf{\textit{f}}_{0} ^{\, n}(p),\eta))\subseteq \mathcal{N}^{u}(  \textbf{\textit{f}}_{0} ^{\, n-1}(p),\eta )$;  
\item  if $q\in \mathcal{N}^{u}( p_{n+1},\delta )$, where $p_{n}=\textbf{\textit{f}}_{0}^{\, n}(p)$, and $k\geq0$ we have  
\begin{equation*} d(\textbf{\textit{f}}_{n+1}^{-(k+1)}(q),\textbf{\textit{f}}_{n+1}^{-(k+1)}(p_{n+1}))\leq \frac{2}{\Delta_{n+1}}\tilde{\lambda}^{k+1}d(q,p_{n+1}). 
\end{equation*}    
\item  $\mathcal{N}^{u}(p,\eta )$ depends continuously on $p.$   
\end{enumerate}
\end{teo}

 \begin{teo}\label{variedadeestavel}    If     $(\textbf{M},\langle\cdot,\cdot\rangle, \textbf{\textit{f}}\,)$ admits a sequence of positive numbers $(\epsilon_{n})_{n\in\mathbb{Z}}$ satisfying   \eqref{cccc}  for    each $p\in M_{0}$, then  there exists a sequence of positive numbers $\eta= (\eta_{n})_{n\in\mathbb{Z}},$ such that    $\mathcal{N}^{s}( \textbf{\textit{f}}_{0} ^{\, n}(p) ,\eta)$  is a differentiable submanifold of $M_{n}$ with:
\begin{enumerate}[\upshape (i)]
\item   $T_{\textbf{\textit{f}}^{n}(p)}\mathcal{N}^{s}( \textbf{\textit{f}}^{\, n}_{0}(p),\eta)=E_{\textbf{\textit{f}}^{n}(p)} ^{s}$;   
\item $f_{n}  (\mathcal{N}^{s}( \textbf{\textit{f}}_{0} ^{\, n}(p),\eta))\subseteq \mathcal{N}^{s}(  \textbf{\textit{f}}_{0} ^{\, n+1}(p),\eta)$; 
\item if  $q\in \mathcal{N}^{s}( p_{n},\epsilon)$, where $p_{n}=\textbf{\textit{f}}_{0}^{\, n}(p)$, and $k\geq1$ we have  
\begin{equation*} d(\textbf{\textit{f}}_{n}^{k}(q),\textbf{\textit{f}}_{n}^{k}(p_{n}))\leq \frac{2}{\Delta_{n}}\tilde{\lambda}^{k}d(q,p_{n}),
\end{equation*} 
\item $\mathcal{N}^{s}(p,\eta )$ depends continuously on $p.$    
\end{enumerate}
\end{teo}

 The present work was carried out with the support of   the Conselho Nacional de Desenvolvimento Cient\'ifico  e Tecnol\'ogico - Brasil (CNPq) and the Coordena\c c\~ao de Aperfei\c coamento de Pessoal de N\'ivel Superior  (CAPES). The author would like to thank the institution   Universidade de S\~ao Paulo (USP)    for their   hospitality and support  during the course of the writing. Special thanks for A. Fisher, my doctoral advisor, who has inspired and aided me along the way.

\end{document}